\documentclass[]{{birkart}}

\usepackage{graphicx}
\usepackage{amsfonts}
\usepackage{amsmath,amscd}
\usepackage{amsthm}

{ \theoremstyle{plain}
\newtheorem{theorem}{Theorem}
\newtheorem{lemma}{Lemma}

}

{ \theoremstyle{definition}
\newtheorem{definition}{{\bf Definition}}}

{ \theoremstyle{definition}

}

{ \theoremstyle{definition}
\newtheorem{remark}{Remark}
}

\newcounter{nw}
\newcounter{nm}
\def\RR{{\Bbb R}}
\def\mod{\theta(A)}
\def\proof{{\bf Proof. }}

\def\R{\mathbb{R}}
\def\prx{\frac{\partial }{\partial x}}
\def\pry{\frac{\partial }{\partial y}}
\DeclareMathOperator{\im}{\mathrm{Im}}
\DeclareMathOperator{\re}{\mathrm{Re}}
\DeclareMathOperator{\sn}{\mathrm{sn}}
\DeclareMathOperator{\cn}{\mathrm{cn}}

\def\ZM{{\bf M}}
\def\ZG{{\bf G}}

\begin{document}

\title[Double-periodic maximal surfaces]{Double-periodic maximal surfaces with
singularities}

%
%

\author{Sergienko Vladimir V.}

\email{}
\thanks{The first author was supported by grant INTAS No.10170}

%
%
\author{Tkachev Vladimir G.}

\thanks{The second author was supported by grant INTAS No.10170 and
Ministerstvo Vyshego Obrazovaniya Rossii N 97-0-1.3.-114.}

%
%

%
%

\begin{abstract}
We construct and study a family of double-periodic almost entire solutions of the maximal surface equation. The solutions are parameterized by a submanifold of $3\times 3$-matrices (the so-called generating matrices). We show that the  constructed solutions are either space-like or of mixed type with the light-cone type isolated singularities.
\end{abstract}

\maketitle

\section{Introduction}

We consider the following equation
\begin{equation}
(1-u_y{}^2)u_{xx}+2u_xu_yu_{xy}+(1-u_x^2)u_{yy}=0.
\label{eq1}
\end{equation}
It is well known then that the graph $z=u(x,y)$ is a zero mean curvature surface in  Minkowski space
$\RR^3_1(x,y,z)$ equipped with the indefinite metric
$$
ds^2=dx^2+dy^2-dz^2.
$$
The graph $M$ in $\RR^3_1(x,y,z)$ given by $z=u(x,y)$ is called  {\it space-like} at the point $(x,y)$ if
\begin{equation}
|\nabla u(x,y)|<1.
\label{eq2}
\end{equation}
For a space-like solution, eq. (\ref{eq1}) is can be written in the divergence equation
\begin{equation}
{\rm div}\frac{\nabla u}{\sqrt{1-|\nabla u|^2}}=0,
\label{eq3}
\end{equation}
called also the maximal surface equation (notice that (\ref{eq3}) is well defined for any dimension).
It is well known that \textit{entire} space-like solutions are trivial.

\begin{theorem}[\cite{Cal}, \cite{ChJ}]
The only entire  space-like $C^2$-regular solutions to (\ref{eq1}) are
affine functions $u(x,y)=ax+by+c$.
\label{th0}
\end{theorem}

On the other hand there are nontrivial (space-like) solutions  which are regular and well-defined everywhere in $\R^{2}$ except possibly for a set of isolated where the absolute value of the gradient attains its maximum value: $|\nabla u|=1$. The classical example is the  maximal catenoid $(\sinh z)^2=x^2+y^2$ in $\R^3_1$ (see Figure~\ref{fig:cat}).

\begin{figure}[h]
\includegraphics[height=0.36\textheight,angle=270,keepaspectratio=true]{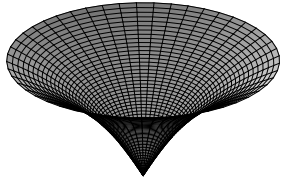}
\vspace*{-1.4cm}
\caption{The maximal catenoid $(\sinh z)^2=x^2+y^2$} \label{fig:cat}
\end{figure}

We call $u(x,y)$  an \textit{almost entire} solution of
(\ref{eq1}) if  $u$ is a continuous function in $\mathbb{R}^2$ and it is $C^2$-regular everywhere in $\mathbb{R}^2$ outside the set
\begin{equation}\label{SSI}
\Sigma_u= \{(x,y): |\nabla u(x,y)|=1\}
\end{equation}
which consists of isolated points only. We point out that we do not impose other constrains like, for example, the space-likeness of $u$. A solution of (\ref{eq1}) is called of the {\it mixed} type if  $|\nabla u|-1$ assumes both signs.

If $u(x,y)$ is a space-like solution in the punctured neighborhood  $\R^2\setminus\{a\}$, $a\in\R^2$,
then it follows from the results\footnote{In fact, the same property holds in any dimension $\R^n$, $n\ge 2$.} due to K.~Ecker \cite{Eck}, V.~Klyachin and V.~Miklyukov \cite{KM1}
that in a neighborhood of $a$ the function $u(x,y)$ asymptotically behaves
like the light cone, there is $\varepsilon^2 = 1$ such that
\begin{equation}
u(x)=u(a)+\varepsilon |x-a|+o(|x-a|), \quad x\rightarrow a.
\label{eq4}
\end{equation}
It would be interesting to study if the last property true for solutions of the mixed
type. In contrast to maximal surfaces,  two-dimensional minimal surfaces in Euclidean space has no isolated singularities in the following sense. If a minimal surface is a graph of a bounded $C^2$-function in a punctured neighborhood then it can be extended to a $C^2$-function, actually, to an analytic one in the whole neighborhood (see \cite{Nit}, \cite{DeGSt}).

Many known properties of the space-like solutions are basically due to the existence of an analogue of the classic Enneper-Weierstrass representation for maximal surfaces. Unfortunately, an analog of the Enneper-Weierstrass representation is impossible for general solutions of (\ref{eq1}). Indeed,  the function
\begin{equation}
u(x,y)=x+h(y)
\label{eq5}
\end{equation}
with $h\in C^2(\RR)$ (but not $C^3$, say) is an entire but non-analytic solution of (\ref{eq1}).
Observe also that in this case we have $|\nabla u|^2 = 1+h'{}^2(y) \geq 1$ everywhere in $\RR^2$.

In this paper we develop a non-parametric method for constructing almost entire solutions to (\ref{eq1}). This method allows us
also to construct solutions of mixed type. More precisely, we obtain a family of doubly-periodic real analytic almost entire solutions, i.e. the solutions satisfying the periodicity condition
$$
u(x+\tau_1n, y+\tau_2m)=u(x,y), \quad n,m\in \mathbb{Z}
$$
for some positive $\tau_1$, $\tau_2$. It is interesting to notice that the our examples have the same light-cone behavior at singular points (where the gradient of the solution has the unit length) even if the singular points have the mixed type. In Section~\ref{raz5} we discuss the structure of the level-sets  $G_u=\{(x,y):|\nabla u(x,y)|=1\}$ in more detail.

We start with two following observations.

\medskip
(i) The first property provides a complete classification of all solutions to (\ref{eq1}) with harmonic level sets and was announced in \cite{SerTk}.

\begin{theorem}\label{th_one}
Let $u=F\circ\phi$ be a solution of (\ref{eq1}), $\phi(x,y)$ be a harmonic function
and $F(t)$ be a twice differentiable function of one variable. Then
$$
\phi(x,y)={\rm Re\,}\int_{}^{}\frac{d\zeta}{g(\zeta)},
$$
where a holomorphic function $g(\zeta)$ is one of the following:
\begin{enumerate}
\item[(i)] $g(\zeta)=a\zeta+c$,
\item[(ii)] $g(\zeta)=a{\rm e}^{b\zeta}$,
\item[(iii)] $g(\zeta)=a\sin(b\zeta+c)$.
\end{enumerate}
Here $\zeta=x+iy\in \mathbb{C}$ and $a^2,b^2\in\RR, \  c\in \mathbb{C}.$
Moreover, in this case the function $F(\phi)$ is found from the equation
$$
F''(\phi)-\frac{{\rm Re}\, g'(\zeta)}{|g(\zeta)|^2}F'^3(\phi)=0.
$$

\end{theorem}

The proof of the theorem is outlined in Section~\ref{the1} below (see \cite{SerTk} for more detailed discussion).
We comment briefly the mentioned in the theorem alternatives. One can readily verify that
the first two alternatives yield the well-known classical examples of a plane, a maximal catenoid,
a helicoid and various analogues of the Scherk's surface in $\mathbb{R}^3_1$.

\vspace*{-1.4cm}
\begin{figure}[h]
\includegraphics[height=0.36\textheight,angle=270,keepaspectratio=true]{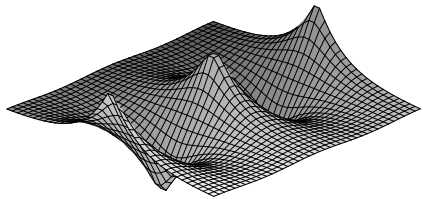}
\vspace*{-1cm}
\caption{A single-periodic surface (\ref{oneper}) for $k=\frac{4}{5}$} \label{fig:oneper}
\end{figure}

More interesting is alternative (iii) that provides for $a,b\in\RR$  a new class of single-periodic (uniformly bounded)
space-like solutions (see Figure~\ref{fig:oneper})
\begin{equation}
{\rm sn}\biggl(\frac{z}{k'};k\biggr)=\frac{\sin x}{\cosh y}.
\label{oneper}
\end{equation}
Here $k\in(0,1)$, $k'=\sqrt{1-k^2}$ and $\sn(t;k)$ denotes the elliptic Jacobi
sinus (see also Section~\ref{raz6} for more details). In fact, this single-periodic solution is a limit cases of the
double-periodic solutions discussed in Section~\ref{sec2per}.

(ii) Another example is the family
\begin{equation}
\sin z=\alpha \sin \frac{x}{\sqrt{\alpha}}+
(1-\alpha) \sin \frac{y}{\sqrt{1-\alpha}}, \qquad \alpha\in(0;1).
\label{sinsin}
\end{equation}
One can easily verify that the above functions are double-periodic solutions of (\ref{eq1}).

\vspace*{-0.7cm}
\begin{figure}[h]
\includegraphics[height=0.36\textheight,angle=270,keepaspectratio=true]{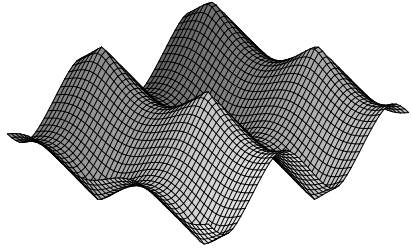}
\vspace*{-1cm}
\caption{A double-periodic surface (\ref{sinsin}) for $\alpha=\frac{1}{4}$} \label{fig:sinsin}
\end{figure}

For $\alpha=1/2$, equation (\ref{sinsin}) (after a suitable rotation in the $xy$-plane) turns into a product
\begin{equation}
\sin z=\sin x \sin y.
\label{sinsin1}
\end{equation}
We shall exploit namely this multiplicative form in our further constructions. To this aim, notice that examples (\ref{oneper}) and (\ref{sinsin1}) can be represented as members of one, more general, family. Indeed, equation (\ref{oneper}) should be regarded as
$$
\sn (\frac{z}{k'},k)=\sn(x,0) \sn(y,1),
$$
and (\ref{sinsin}) for $\alpha=1/2$ can be written as
$$
\sn (z,0)=\sn(x,0) \sn(y,0),
$$
where $\sn (t,k)$ is the elliptic Jacobi sinus.

\section{Harmonic level solutions}
\label{the1}

Here we outline the proof of Theorem~\ref{th_one} (see however \cite{SerTk} for a more detailed discussion).
We consider solutions to (\ref{eq1}) given in the form $u(x,y)=F(\phi (x,y))$,
where $F=F(t)$ and $\phi(x,y)$~ are some $C^2$-regular functions. A simple calculation reveals that
(\ref{eq1}) is equivalent to the equation
\begin{equation}
A(x,y)F''+B(x,y)F'+C(x,y)F'^3=0,
\label{ur2}
\end{equation}
where $F'=F'(t)$,  $A(x,y)={\phi }^2_x+{\phi }^2_y$, $B(x,y)={\phi
}_{xx}+{\phi }_{yy}$, and
$$
C(x,y)=-{\phi }^2_x{\phi
}_{yy}+2\phi_x\phi_y\phi_{xy}-{\phi
}^2_y{\phi}_{xx}.
$$

\begin{lemma}
Let $f(w)$, $g(w)$ be holomorphic functions of $w=x+iy$. Then the
following identities take place
\begin{equation}
\begin{split}
\prx \re (f \bar g)&=\re (f'\bar g +f \bar g'),\\
\pry \re (f \bar g)&=-\im (f'\bar g -f \bar g').
\label{l1}
\end{split}
\end{equation}
Here $\bar g$ denotes the conjugate to $g$ function.
\label{lemm1}
\end{lemma}

\begin{proof} The Cauchi-Riemann conditions imply
\begin{equation}
\begin{split}
\prx \re f=\re f',& \qquad \pry \re f=-\im f',\\
\prx \im f=\im f',&  \qquad \pry \im f=\re f'.
\label{koshi}
\end{split}
\end{equation}
Then
\begin{equation*}
\begin{split}
\prx \re (f \bar g)&=\prx (\re f \re g+\im f \im g)
=\re f' \re g+\re f \re g'\\
&+\im f' \im g+\im f \im g'= \re (f'\bar g +f \bar g')
\end{split}
\end{equation*}
and the first identity of the lemma follows. Similar one gets the remained identities.
\qed
\end{proof}

Next, by our assumption the function $\phi (x,y)$ is harmonic function, hence we have $\phi (x,y)=\re h(w)$ for some
holomorphic function $h(w)=h(x+iy)$.
In order to find the coefficients $A$, $B$ and $C$ in (\ref{ur2}),
we apply (\ref{koshi}):
$$
\phi _x=\prx \re h=\re h', \qquad
\phi _y=\pry \re h=-\im h',
$$
and $\phi _{xx}=\re h''$, $\phi _{xy}=-\im h''$, $\phi_{yy}=-\re h''$. Then
\begin{equation*}
\begin{split}
A(x,y)&={\phi }^2_x+{\phi }^2_y=|h'(w)|^2,
\\
B(x,y)&={\phi }_{xx}+{\phi }_{yy}=0,
\\
C(x,y)&=\re (h''{\bar {h'}}^2)
\end{split}
\end{equation*}
and the equation~(\ref{ur2}) takes the form
$$
|h'|^2F''+\re (h''
\bar{h'}^2){F'}^3=0.
$$
Setting
\begin{equation}\label{recall}
g(w)\equiv 1/h'(w)
\end{equation}
we find finally
\begin{equation}
\frac{F''(\phi)}{F'^3(\phi)}=\frac{\re g'}{|g|^2}.
\label{chast1}
\end{equation}

Thus Theorem~\ref{th_one} follows from the lemma below.

\begin{lemma}
The right hand side  $|g|^{-2}\,\re g'$ in (\ref{chast1}) is a function of $\phi =\re h(w)$  if and only if the following identity holds:
\begin{equation}
gg''-{g'}^2=c, \qquad c\in \mathbb{R}.
\label{mmm}
\end{equation}
\label{lemm2}
\end{lemma}

\begin{proof} Denote $\psi (x,y)=|g|^{-2}\,\re g'$. If $\psi$ and $\phi$ are functional dependent then their Jacobi determinant vanishes:
$$
\frac{\partial (\phi ,\psi)}{\partial (x,y)}=0.
$$
We show that this is equivalent to (\ref{mmm}).
Indeed, we have $\phi _x=\re h'$, $\phi
_y=-\im h'$, and (\ref{l1}) readily yields
\begin{equation}
\begin{split}
\psi _x &=2\re
(h''\bar h')\re g'+\re (h'\bar h'g''),\\
\psi _y &=-2\im
(h''\bar  h')\re g'-\im (h'\bar h'g'').
\label{kos1}
\end{split}
\end{equation}
Then we find the Jacobian:
\begin{equation}
\begin{split}
0=\phi
_x\psi _y-\phi _y\psi _x=-2\im (h'' \bar{h'}^2)\re g'-\im (g''h'\bar{h'}^2),
\end{split}
\end{equation}
therefore
$$
-2|h'|^4\im \left( \frac{h''}{{h'}^2}\right)\re g'-|h'|^4\im \left(
g''\frac{1}{h'}\right)=0.
$$
Recalling that (\ref{recall}) we find
$$
-2\im g'\re g' +\im (g'' g)=\im (g''g-{g'}^2)=0.
$$
Notice that in the last expression $gg''-{g'}^2$ is a holomorphic function whose imaginary part vanishes identically on an open set. It follows
that the holomorphic function is identically zero, hence there is a
real constant $c$ such that $gg''-{g'}^2\equiv c$. The lemma is proved.
\qed
\end{proof}

\section{One-periodic solutions}

An easy analysis shows that  the differential equation (\ref{mmm}) has the following solutions:

\begin{enumerate}
\item[(i)] $g(\zeta)=a\zeta+c$,
\item[(ii)] $g(\zeta)=a \,e^{b\zeta}$,
\item[(iii)] $g(\zeta)=a\sin(b\zeta+c)$.
\end{enumerate}

We leave it to the reader to verify that the cases (i), (ii) reduces to the classic examples:
plane, the maximal catenoid, the helicoid and
a maximal analogue of Sherk's surface.

Now we consider  $g(w)=\sin w$. Here we have
$h'(w)=\frac{1}{g(w)}$, hence,
$$
h(w)=\frac{1}{2}\ln \frac{\cos w -1}{\cos w+1}+C.
$$
Without loss of generality, we may assume that the constant $C$ in the
right side is zero. Then
\begin{equation}\label{expl}
\phi (x,y)= \re h(w)=\frac{1}{2}\ln \left
|\frac {\cos w -1}{\cos w+1} \right |=\frac{1}{2}\ln\frac{\cosh y-\cos
x}{\cosh y+\cos x},
\end{equation}
and
$$
\frac{1}{|g|^2}\re g'=\frac{\re \cos w }{|\sin w|^2}.
$$
On the other hand, we have
$$
2\sinh 2\phi (x,y)=\left
|\frac{\cos w -1}{\cos w+1} \right |-\left |\frac{\cos w +1}{\cos
w-1}\right | =
-\frac {4\re \cos w}{|\sin w|^2}=-\frac{4}{|g|^2}\re g'.
$$
Thus, (\ref{chast1}) takes the form
$$
F''(\phi)+\frac{1}{2}F'^3(\phi)\sinh 2 \phi =0.
$$
Intergrating this ordinary equation yields
$$
\frac{1}{{F'}^2(\phi)}=\frac{1}{2}\cosh 2\phi+\frac{k}{2}, \quad k
\in \mathbb{R}.
$$
Hence
\begin{equation}\label{varvar}
F'(\phi)=\frac{1}{\sqrt{\frac{1}{2}\cosh 2\phi
+\frac{k}{2}}}=\frac{2e^{\phi }}{\sqrt{e^{4\phi }+2ke^{2\phi }+1}}.
\end{equation}

Now we find the values of parameter $k$ which corresponds
to \textit{space-like} $F$, that is where the inequality
$$|\nabla F(\phi (x,y))|<1$$ holds everywhere when $F(\phi)$ is regular.
To  this aim we notice that
\begin{equation}
|\nabla F(\phi (w))|=|F'(\phi)|\,|\nabla \phi|=
|F'(\phi)|\,|h'(w)|.
\label{ggg}
\end{equation}
By introducing a new variable $\gamma =\frac{\cos x}{\cosh y}$,
we find
$$
|h'(w)|=\frac{1}{|\sin w|}=\frac{1}{\sqrt{\cosh^2y-\cos^2x}}=
\frac{1}{\cosh y\sqrt{1-\gamma^2}}.
$$

On the other hand, by virtue of (\ref{expl})
we find
$$
|F'(\phi)|=\sqrt{\frac{2(1-\gamma^2)}{(1+k)-(k-1)\gamma^2}}.
$$
Plugging the above expression into (\ref{ggg}) yields
$$
|\nabla F(\phi (w))|=\frac{1}{\cosh
y}{\frac{\sqrt2}{\sqrt{(1+k)-(k-1)\gamma^2}}}.
$$
A simple analysis of the last expression shows that the space-like
condition for $F(\phi)$ (for any $x$ and $y$) is fulfilled only for $k>1$.

In order to find $F$ explicitly we integrate (\ref{varvar}). We assume without less of generality that $F(0)=0$. Then
\begin{equation*}
\begin{split}
F(\phi)&=2\int\limits _{1}^{e^{\phi}} \frac{d\,
\xi}{\sqrt{{\xi}^4 +2k{\xi}^2+1}}\\
&=\sqrt{2}\int\limits _{0}^{\tanh \phi} \frac{d\,
t}{\sqrt{1-{t}^2}\sqrt{(1+k)-(k-1){t}^2}}\\
&=\alpha '\int\limits _{0}^{\tanh \phi} \frac{d\,
t}{\sqrt{1-{t}^2}\sqrt{1-{\alpha}^2{t}^2}},
\end{split}
\end{equation*}
where  ${\alpha}^2=\frac{k-1}{k+1}$ and ${\alpha '}^2=1-{\alpha }^2=1-\frac{k-1}{k+1}=\frac{2}{k+1}$.
The latter integral can be simplified by means of the Jacobi elliptic sinus, namely
$$
\sn \left( \frac{F(\eta)}{\alpha '};\alpha \right)=\tanh \eta
$$

Thus, finally  we get the solution $z=F(\phi(x,y))$ in the following form:
$$
\sn \left( \frac{z}{\alpha '};\alpha \right)=\tanh \phi(x,y)=
-\frac{\cos x}{\cosh y}.
$$

In summary:

\begin{theorem}
Given an arbitrary $a>0$ and $\alpha \in (0;1)$, $\alpha '=\sqrt{1-{\alpha}^2}$, the implicit equation
$$
\sn \left( \frac{az}{\alpha '};\alpha \right)=\frac{\cos ax}{\cosh ay},
$$
defines a space-like solution $z=u_{a,\alpha}(x,y)$ of (\ref{eq1}). This solution is a real
analitic function everywhere in $\mathbb{R}^2$ outside the set of points
$$
\{x=\frac{\pi k}{a},\; y=0, \quad k\in \mathbb{Z}\}.
$$
\end{theorem}

Some  geometric remarks are appropriate here. The dependence on $a$ amounts to a homothety of a graph $z=u_{a,\alpha}(x,y)$, while
dependence on $\alpha$ has a more intrinsic nature. In particular, for different values of $\alpha\in(0,1)$, the corresponding graphs $z=u_{1,\alpha}(x,y)$ do not Lorentz isometric.

One can also see that $M(\alpha)$ is located in the parallel slab
$|z|\leq \mathrm{K}(\alpha)\alpha'$, where $\mathrm{K}(\alpha)$ is the complete elliptic integral of the first kind (the least positive solution of equation $\sn(\mathrm{K}(\alpha),\alpha)=1$).

\section{Generating Matrices}
\label{sec:genmat}

In this section we introduce a family of generating matrices which will be used in constructing of doubly-periodic solutions.
In what follows, by a `matrix' we mean a $3\times3$-matrix with real coefficients. By $S$ we denote also the class of all
permutations of the index set $\{1,2,3\}$. Below we define the generating
matrix and consider its basic properties.

A non-zero matrix $A=(a_{ij})$ is called \textit{{generating}}, or $A\in \ZM _3$, if
for all $(i,j,k)$ and $(\alpha,\beta,\gamma)\in S$ the following identities
hold
\begin{equation}
a_{i\alpha}a_{i\beta}=a_{j\gamma}a_{k\gamma}.
\label{eq21}
\end{equation}

Let us associate with the matrix $A=(a_{ij})$  a new matrix $A'$ with entries
\begin{equation}
A'=\left(
\begin{array}{ccc}
a_{11}&a_{22}&a_{33}\\
a_{23}&a_{31}&a_{12}\\
a_{32}&a_{13}&a_{21}
\end{array}
\right).
\label{eq22}
\end{equation}
Then (\ref{eq21}) is equivalent to vanishing of all second order minors of $A'$. Since $A$ is non-zero, $A$ is generating if and only if ${\rm rank}\,A'=1.$
The last observation yields that $A$ must be of the form
\begin{equation}
A=\left(
\begin{array}{ccc}
p_1p_2&q_1r_2&r_1q_2\\
r_1r_2&p_1q_2&q_1p_2\\
q_1q_2&r_1p_2&p_1r_2
\end{array}
\right),
\label{eq23}
\end{equation}
for an appropriate set of reals $p_i,q_i,r_i$ ($i=1,2$) satisfying the non-degenerating property
$$
p_i^2+q_i^2+r_i^2\neq 0, \  i=1,2.
$$

Indeed, the above rank condition yields the existence of a non-zero
vector $\xi=(p_1,q_1,r_1)$ and real scalars  $p_2,q_2,r_2$, not all zero, such that every row of $A'$ is collinear to $\xi$
with $p_2,q_2$ and $r_2$ being the proportionality  coefficients. This implies (\ref{eq23}).

Another useful property of generating matrices is that \textit{the product of all elements in each row and each string has the same value}.  We call the common value  the {\it module} of
the generating matrix $A$ and denote it by $\mod$. If $\mod\neq 0$ we call $A$
{\it elliptic}, otherwise $A$ is called {\it parabolic}. The following proposition follows easily from
the mentioned above representations and characterizes parabolic generating matrices completely.

\begin{lemma}
After an appropriate permutations of its strings and rows, any
parabolic generating matrix  $A\in \ZM _3$ can be brought into one of the following forms:
\begin{equation*}
\begin{split}
\left(
\begin{array}{ccc}
0&a_{12}&a_{13}\\
a_{21}&0&a_{23}\\
a_{31}&a_{32}&0
\end{array}
\right), \,
\left(
\begin{array}{ccc}
0&a_{12}&a_{13}\\
a_{21}&0&0\\
a_{31}&0&0
\end{array}
\right), \,
\left(
\begin{array}{ccc}
a_{11}&0&0\\
0&a_{22}&0\\
0&0&a_{33}
\end{array}
\right).
\end{split}
\end{equation*}
Here either all $a_{ij}\neq 0,$ or $A$ contains a zero-row or zero-string.
\end{lemma}

\begin{lemma}
If $A$ is a generating matrix then the quantity
\begin{equation}
\Delta_\alpha:=
(a_{j\alpha}+a_{k\alpha}-a_{i\alpha})^2-4a_{i\beta}a_{i\gamma},
\end{equation}
where $(i,j,k),(\alpha,\beta,\gamma)\in S$, depends on the index $\alpha$
only.
\end{lemma}

\begin{proof}
We have
\begin{equation*}
\begin{split}
\Delta_\alpha&=a_{j\alpha}^2+a_{k\alpha}^2+a_{i\alpha}^2+
2(a_{j\alpha}a_{k\alpha}-2a_{i\beta}a_{i\gamma}-a_{j\alpha}a_{i\alpha}-
a_{k\alpha}a_{i\alpha})\\
&=a_{j\alpha}^2+a_{k\alpha}^2+a_{i\alpha}^2-
2(a_{j\alpha}a_{k\alpha}+a_{j\alpha}a_{i\alpha}+a_{k\alpha}a_{i\alpha})\\
&=\sum_{\sigma=1}^{3}a_{\sigma\alpha}^2 -
\sum_{\sigma\neq\rho}^{}a_{\sigma\alpha}a_{\rho\alpha}.
\end{split}
\end{equation*}
Now it is evident that the last expression depends only on $\alpha$ and the
lemma is proved.
\qed
\end{proof}

\begin{definition}
The quantity $\Delta(A)\equiv \frac{1}{4}\Delta_2$ will be called the {\it discriminant}
of $A$.
\end{definition}

Define an action of the multiplicative group  $\RR^+\times\RR^+$ ($\RR^+=\{z\in\R{}: z>0\}$)
on $3\times 3$ matrices as follows
\begin{equation}\label{eq:array}
\lambda(A)=\left(
\begin{array}{ccc}
\frac{1}{\lambda_1}a_{11}&a_{12}&\lambda_1a_{13}\\
\frac{1}{\lambda_2}a_{21}&a_{22}&\lambda_2a_{23}\\
\lambda_1\lambda_2a_{31}&a_{32}&\frac{1}{\lambda_1\lambda_2}a_{33}
\end{array}
\right),
\end{equation}
where
$$
\lambda=(\lambda_1,\lambda_2),\quad \mu=(\mu_1,\mu_2).
$$
It is evident that $\lambda(A)\in \ZM _3$ if and only if $A\in \ZM _3$. Denote by $\ZG_3=\ZM_3/(\RR^+\times\RR^+)$ the factor space and shall write $A\sim B$ if $B=\lambda(A)$. One can  verify the module $\theta(A)$ and the discriminant $\Delta(A)$ give rise to the $\ZG_3$. The reason why this factor space is important we shall see below: two  equivalent matrices $A\sim B$ produce one maximal surface.

\begin{lemma}
Let $A\in \ZM _3$ be an elliptic generating matrix ($\mod\ne0$).
Then  there is $\varepsilon_k=\pm1$, $k=2,3$, such that $A$ is equivalent to the matrix
\begin{equation}
A_{\varepsilon_2\varepsilon_3}(a,b,c)=
\left(
\begin{array}{ccc}
a&b&c\\
\varepsilon_2\varepsilon_3c&\varepsilon_2a&\varepsilon_3b\\
\varepsilon_2\varepsilon_3b&\varepsilon_2c&\varepsilon_3a
\end{array}
\right).
\label{eq27}
\end{equation}
\end{lemma}

\begin{proof}
Let us denote by $\varepsilon_k={\rm sign}(a_{kk}/a_{11}), \  k=2,3,$
and set $\lambda_1=\varepsilon_2a_{11}/a_{22}, \
\lambda_2=\varepsilon_3a_{33}/a_{11}$.
Then  $\lambda_i>0$ and for $a=\varepsilon_2a_{22}$, $b=a_{12}$,
$c=\varepsilon_3a_{32}$ we conclude that $\lambda(A)$ has the form (\ref{eq27}).
\qed
\end{proof}

\begin{remark}
It follows from the latter representation that the map
$$
(a,b,c,\varepsilon_2,\varepsilon_3) \rightarrow \ZG_3
$$
is a well-defined parametrization of the generating factor $\ZG_3$.
\end{remark}


\section{Construction of solutions}
\label{sec2per}

We consider solutions $z(x,y)$ to (\ref{eq1}) which
given in the following implicit form:
\begin{equation}
\zeta(z(x,y))=\phi(x)\psi(y),
\label{eq31}
\end{equation}
where $\phi,\psi$ and $\zeta$ are some $C^2$-regular functions.
Using (\ref{eq31}), we find
\begin{equation}
\zeta'(z)z_x=\phi'\psi, \quad \zeta'(z)z_y=\phi\psi'
\label{eq32}
\end{equation}
and
\begin{equation}
\label{three}
\zeta''z_x^2+\zeta'z_{xx}=\phi''\psi, \quad
\zeta''z_xz_y+\zeta'z_{xy}=\phi'\psi', \quad
\zeta''z_y^2+\zeta'z_{yy}=\phi\psi''.
\end{equation}
Multiplying  componentwise the latter three equations (\ref{three}) by
$$
(1-z_y^2)\zeta'{}^2 \equiv \zeta'{}^2 - \phi^2\psi'{}^2,
$$
$$
2z_xz_y\zeta'{}^2 \equiv 2\phi\psi\phi'\psi',
$$
$$
(1-z_x^2)\zeta'{}^2 \equiv \zeta'{}^2 - \phi'{}^2\psi^2
$$
and taking the sum, we obtain by virtue of (\ref{eq1}) that
$$
\zeta''\zeta'{}^2(z_x^2+z_y^2) = \zeta'{}^2(\phi''\psi+\phi\psi'')
(\phi\phi''\psi'{}^2-2\phi'{}^2\psi'{}^2 +
\phi'{}^2\psi\psi'')\phi\psi.
$$
Applying $\phi\psi=\zeta$ and (\ref{eq32}), we arrive at
\begin{equation}
(\phi''\psi+\phi\psi'')\zeta'^2 -
(\phi\phi''\psi'^2 - 2\phi'^2\psi'^2 +
\phi'{}^2\psi\psi'')\zeta - (\phi'{}^2\psi^2+\phi^2\psi'{}^2)\zeta''=0.
\label{eq34}
\end{equation}

\smallskip
\begin{lemma}\label{lem_relax}
Let $z(x,y)$ be a solution of (\ref{eq1}) given in the form (\ref{eq31}).
Assume that the following conditions hold
\begin{enumerate}
\item[(i)] $\phi$ and $\psi$ assume zero values at some points;\\
\item[(ii)] there exist $C^1$-functions $P$, $Q$ and $H$ such that
\begin{equation}
\phi'{}^2=P(\phi^2), \quad \psi'{}^2=Q(\psi^2), \quad
\zeta'{}^2=H(\zeta^2).
\label{eq35}
\end{equation}
\end{enumerate}
Then $P$ and $Q$ are quadratic polynomials.
\end{lemma}

\begin{proof}
We find  from (\ref{eq35}) that
$$
\phi''=P'(\phi^2)\phi, \quad \psi''=Q'(\psi^2)\psi, \quad
\zeta''=H'(\zeta^2)\zeta,
$$
hence, by virtue of (\ref{eq34}) and (\ref{eq31}) we find
\begin{equation}\label{seee}
\begin{split}
&H(\phi^2\psi^2)\left[P'(\phi^2)+Q'(\psi^2)\right] -
H'(\phi^2\psi^2)\left[\psi^2P(\phi^2)+\phi^2Q(\psi^2)\right]\\
&-\left(\phi^2P'(\phi^2)Q(\psi^2) - 2P(\phi^2)Q(\psi^2)
+\psi^2Q'(\psi^2)P(\phi^2)\right) = 0.
\end{split}
\end{equation}
The latter identity holds for any admissible values of $\phi$ and $\psi$ (here regarded as independent variables). By
our assumption, there are $x_0$ and $y_0$ such that
$\phi(x_0)=0$ and $\psi(y_0)=0$. Substituting by turn $\phi=0$ and $\psi=0$ we find two following identites:
\begin{equation}
\begin{split}
&(h_0-p_0v)Q'(v)+2p_0Q(v)+h_0p_1-h_1p_0v=0,\\
&(h_0-q_0u)P'(u)+2q_0P(u)+h_0q_1-h_1q_0u=0,
\end{split}
\label{eq36}
\end{equation}
where $v=\psi^2$, $u=\phi^2$, $p_0=P(0)$, $p_1=P'(0)$, $q_0=Q(0)$, $q_1=Q'(0)$,
$h_0=H(0)$, $h_1=H'(0)$.

If $p_0=0$ ($q_0=0$ resp.) then the function
$Q(v)$ ($P(u)$ resp.) is linear. In the remained case, one can solve the ordinary
differential equations (\ref{eq36}) to obtain
$$
P(u)=au^2+\left(h_1-2a\frac{h_0}{q_0}\right)u-\frac{1}{2q_0}\left(h_0(h_1+q_1)-
2a\frac{h_0^2}{q_0}\right),
$$
$$
Q(v)=bv^2+\left(h_1-2b\frac{h_0}{p_0}\right)v-\frac{1}{2p_0}\left(h_0(h_1+p_1)-
2b\frac{h_0^2}{p_0}\right),
$$
where $a,b\in\RR$ some real constants. The lemma is proved.
\qed

\end{proof}

\begin{remark}
The general solution to (\ref{eq34}) becomes  more extensive if we relax condition (i) in Lemma~\ref{lem_relax}.
For example, by using exponential, both the rotationally symmetric solutions $(\sinh z)^2=x^2+y^2$ and surfaces
of the form (\ref{eq5}) can be brought into the form (\ref{eq31}). But as already mentioned, the former functions
are allowed to be of a lower regularity class in contrast to the real analyticity of solutions to (\ref{eq35}).
\end{remark}

By Lemma~\ref{lem_relax}, we can consider
$\phi$, $\psi$ and $\zeta$ satisfying to the following elliptic equations:
\begin{equation}
\begin{split}
\phi'{}^2&=a_1-2b_1\phi^2+c_1\phi^4,\\
\psi'{}^2&=a_2-2b_2\psi^2+c_2\psi^4,\\
\zeta'{}^2&=c_3+2b_3\zeta^2+a_3\zeta^4,
\end{split}
\label{eq37}
\end{equation}
with $a_i$, $b_i$ and $c_i$ $(i=1,2,3)$ to be chosen later.

Now we show that every solution of (\ref{eq1}) satisfying
(\ref{eq31}) and (\ref{eq37}) can be associated with a certain
generating matrix $A\in \ZM_3$. To this end, we notice that
\begin{equation}
\phi''=2\phi(c_1\phi^2-b_1),\ \psi''=2\psi(c_2\psi^2-b_2),\
\zeta''=2\zeta(a_3\zeta^2+b_3).
\label{eq38}
\end{equation}
By virtue of (\ref{eq37}) and (\ref{eq38}) we obtain from
(\ref{seee}) that
\begin{equation*}
\begin{split}
&(c_1c_2-a_3b_1-a_3b_2)\phi^4\psi^4+(a_3a_2-c_1b_3-c_1b_2)\phi^4\psi^2+
(a_3a_1-b_1c_2-c_2b_3)\phi^2\psi^4\\
&+ (a_1b_3+a_1b_2-c_2c_3)\psi^2 +
(a_2b_1-c_1c_3+a_2b_3)\phi^2 + (b_2c_3+b_1c_3-a_1a_2) = 0.
\end{split}
\label{eqsplit}
\end{equation*}
Since $\phi$ and $\psi$ are algebraically independent,
all the coefficients of $\phi^i\psi^j$ are zero.
Let us write $$\beta_i=b_j+b_k,\qquad (i,j,k)\in S.$$
Then the following matrix satisfies  the generating conditions
\begin{equation}
A=\left(
\begin{array}{ccc}
a_1&\beta_1&c_1\\
a_2&\beta_2&c_2\\
a_3&\beta_3&c_3
\end{array}
\right).
\label{eq39}
\end{equation}
Thus, we have proved:

\begin{theorem}
Let $z(x,y)$ be implicitly defined by
$\zeta(z(x,y))=\phi(x)\psi(y)$ with $\phi$, $\psi$ and $\zeta$
satisfying (\ref{eq37}). Then $z(x,y)$ is a solution of (\ref{eq1}) if and
only if the matrix (\ref{eq39}) is generating.
\label{th31}
\end{theorem}

\begin{remark}
We do not distinguish particular solutions to (\ref{eq37}) because they generate two surfaces which differ only by a shift in $\R_1^3$.
\end{remark}

\begin{remark}\label{rem3}
Notice that if a surface is given implicitly by (\ref{eq31}) with  $\phi$, $\psi$ and $\zeta$ satisfying the system (\ref{eq37}) with a generating matrix $A$ then the new set $\alpha_1\phi$, $\alpha_2\psi$ and $\alpha_1\alpha_2\zeta$ with a generating matrix $A_1$ results to the same surface. Recalling the equivalence relation defined in Section~\ref{sec:genmat}, we find that the new generating matrix is $A_1=\lambda(A)$, where $\lambda=(\alpha_1^{-2},\alpha_2^{-2})$. This shows that  two equivalent elements in $\ZM_3$ generate one surface.
\end{remark}

\begin{lemma}
The discriminant $\Delta(A)$ of the generating matrix~(\ref{eq39})
has the following forms
$$
\Delta(A)\equiv b_i^2-a_ic_i=-(b_1b_2+b_3\beta_3),
$$
where $b_i=\frac{1}{2}(\beta_j+\beta_k-\beta_i)$ for any $(i,j,k)\in S$.
In particular, it doesn't depend on $i=1,2,3;$
\label{lem32}
\end{lemma}
\begin{proof} The identity  immediately follows from the definition of the
discriminant $\Delta(A)$. Taking into account the generating matrix
definition~(\ref{eq39}) we obtain
$$
\Delta(A)=b_1^2-a_1c_1=b_1^2-\beta_2\beta_3=b_1^2-(b_1+b_3)\beta_3=
b_1(b_1-\beta_3)-b_3\beta_3=-b_1b_2-b_3\beta_3.
$$
\end{proof}

\begin{remark}
We notice that for non-zero discriminant $\Delta(A)$, the equations (\ref{eq37}) define the Jacobi elliptic functions. These functions are real analytic and single-periodic for the real values of variables $x,y$ and $z$. Some concrete examples can be found in Section~\ref{raz6} below.
\end{remark}

\begin{remark}\label{remtang}
We finish this section by showing that solutions with zero discriminant $\Delta(A)$ can not have isolated singularities. Indeed, if $\Delta(A)=0$,  Lemma~\ref{lem32} implies that all expressions in the right hand side of (\ref{eq37}) are perfect squares, so that (\ref{eq37}) is equivalent to a simpler system
\begin{equation}
\begin{split}
\phi'{}&=\alpha_1-\gamma_1\phi^2,\\
\psi'{}&=\alpha_2-\gamma_2\psi^2,\\
\zeta'{}&=\gamma_3+\alpha_3\zeta^2,
\end{split}
\label{eq371}
\end{equation}
where $\alpha_i^2=a_i$, $\gamma_i^2=c_i$, $b_i=\alpha_i\gamma_i$ with
the generating conditions
\begin{equation}\label{tanh}
\frac{1}{b_1}+\frac{1}{b_2}+\frac{1}{b_3}=0, \qquad b_1b_2b_3=-\alpha_1^2\alpha_2^2\alpha_3^2
\end{equation}
Solutions of (\ref{eq371}) are hyperbolic or trigonometric tangents, which results to surfaces of Scherk's type. We confine yourself by mentioning only a particular example of the above family for which the resulting surface is an entire solution
\begin{equation}\label{tang}
\tanh x\tanh y = \tanh \frac{z}{\sqrt{2}}.
\end{equation}
\end{remark}
This is a mixed type solution without singularities. The critical level set $|\nabla u|=1$ consists of union of hyperbola-like curves
$$
\tanh^2 x \tanh^2 y +1=2\tanh^2 x, \quad \tanh^2 x \tanh^2 y +1=2\tanh^2 y.
$$

\section{Singular points of two-periodic solutions}

In this section we treat isolated singularities of solutions defined by (\ref{eq31}). In what follows we shall assume that $\Delta(A)\ne0$ (see Remark~\ref{remtang} for zero discriminant solutions).

Let $u$ be  a solution generated by some generating matrix $A\in \ZM_3$ given by (\ref{eq39}). Then the graph $z=u(x,y)$ is defined by the implicit equation
\begin{equation}\label{deF}
F(x,y,z)\equiv \zeta(z)-\phi(x)\psi(y)=0,
\end{equation}
where $\phi$, $\psi$ and $\zeta$ are defined by (\ref{eq37}).
Let $(x_0,y_0,z_0)$ be a point on the graph, i.e. $F(x_0,y_0,z_0)=0$. By the inverse function theorem, the sufficient condition for $u$ to be well-defined and smooth (in fact real analytic) in a neighborhood of the point $(x_0,y_0,z_0)$ is  $\frac{\partial F}{\partial z}(x_0,y_0,z_0)\ne 0$.
We call a point $(x_0,y_0,z_0)$ \textit{special} if  $$\frac{\partial F}{\partial z}(x_0,y_0,z_0)=F(x_0,y_0,z_0)=0.$$
This then can be rewritten as
\begin{equation}
\begin{split}
c_3+2b_3\zeta^2(z_0)+a_3\zeta^4(z_0)&=0, \qquad \zeta(z_0)=\phi(x_0)\psi(y_0),
\end{split}
\label{eq42}
\end{equation}
Notice that a special point need not be \textit{a priori} a genuine  singularity.
We call a special point $m_0=(x_0,y_0)$ \textit{nonremovable singularity} if $u$ is a $C^2$-function in a punctured neighbourhood of $m_0$ but it has no a $C^2$-continuation at $m_0$.

\begin{theorem}
Let  $m_0=(x_0,y_0)$ be an isolated nonremovable singularity of  $u$
associated with the generating matrix $A$ given by (\ref{eq39}). Then $\Delta(A)>0$ and $\phi'(x_0)=\psi'(y_0)=\zeta'(z_0)=0$ (where $z_0=u(x_0,y_0)$). Moreover,
\begin{equation}\label{mmmm}
\phi^2(x_0)=\frac{b_1+\delta\sqrt{\Delta(A)}}{c_1}, \qquad
\psi^2(y_0)=\frac{b_2+\delta\sqrt{\Delta(A)}}{c_2},
\end{equation}
where $\delta^2=1$.
\label{th41}
\end{theorem}

\begin{proof}
Since $m_0$ is a special point, equation (\ref{eq42}) has real roots, thus, its discriminant $b_3^2-a_3c_3$ is non-negative. By Lemma ~\ref{lem32} we conclude that $\Delta(A)\equiv b_3^2-a_3c_3\geq 0$. By our agreement, $\Delta(A)\ne0$, therefore $\Delta(A)>0$.

First we prove that $\phi(x_0)\psi(y_0)\neq 0$. Indeed, let for instance  $\phi(x_0)=0$.
Then
$$
\zeta(z_0)=\phi(x_0)\psi(y_0)=0,
$$
and by (\ref{eq42}): $c_3=0$. Since $\Delta(A)>0$ we obtain $b_3\ne 0$.
Integrating $\zeta'{}^2=2b_3\zeta^2+a_3\zeta^4$ we find that $\zeta(z)$ can be one of the following functions:
$e^{\mu z}$, $1/\sin(\mu z+\lambda)$, $1/ \sinh(\mu z+\lambda)$ or $1/\cosh(\mu
z+\lambda)$. But all these functions have no zeros which contradicts to $\zeta(z_0)=0$. Hence $\phi(x_0)\psi(y_0)\neq 0$.

Now we claim that
\begin{equation}\label{short}
\phi'(x_0)=\psi'(y_0)=0.
\end{equation}
Indeed, since $m_0$ is an isolated singularity there is a punctured neighborhood of $m_0$ where $z=u(x,y)$ defined by (\ref{deF}) is a smooth function which can not be extended to a $C^2$-function at $m_0$. Assume for instance $\phi'(x_0)\ne0$. Then this inequality together with the proven above $\psi(y_0)\ne 0$ implies (by the inversion function theorem) that the equation
$$
\zeta(z)=\phi(x)\psi(y)
$$
defines a smooth two-dimensional surface $U$  in $\R^3$ near $(x_0,y_0,z_0)$ which is a graph with respect to the $x$-direction. On the other hand,  this surface coincides with our solution $z=u(x,y)$ in an (eventual smaller) punctured neighborhood of $m_0$. But in that case, $u(x,y)$ admits a smooth continuation at $m_0$ (given be the surface $U$). The contradition shows that (\ref{short}) is true. Combining (\ref{eq37}) and (\ref{short}) we obtain (\ref{mmmm}), where $\delta^2=1$.


\qed
\end{proof}

\medskip
The following theorem shows that at an isolated nonremovable singularity the constructed surfaces  have the light cone behaviour.
We shall  assume without loss of generality that the singularity is located at the origin: $x_0=y_0=z_0=0$.

\begin{theorem}
Let  $(0,0)$ be an isolated nonremovable singularity of  $z=u(x,y)$
associated with the generating matrix $A$ given by (\ref{eq39}). Then
\begin{equation}\label{assym}
u(x,y)=\delta \sqrt{x^2+y^2}+O(x^2+y^2),
\quad (x,y)\rightarrow (0,0),
\end{equation}
where $\delta^2=1$.
\label{th42}
\end{theorem}

\begin{proof}
We have $\phi(0)\psi(0)=\zeta(0)$ and by Theorem~\ref{th41},
$$\phi'(0)=\psi'(0)=\zeta'(0)=0.
$$
We have as an immediate consequence of (\ref{eq37}) that  $\phi^{(k)}(0)=\psi^{(k)}(0)=\zeta^{(k)}(0)=0$ for odd $k$.
Finding then the Taylor expansion of the second order for (\ref{deF}) near the point $(0,0,0)$ we obtain
\begin{equation}
\begin{split}
\frac{z^2\zeta''_0}{2}-\frac{\psi_0\phi''_0x^2+
\phi_0\psi''_0y^2}{2}=
z^4h(z)+O((x^2+y^2)^2),
\label{eq44}
\end{split}
\end{equation}
where $h(z)$ is a bounded function at $z=0$  and  we denote $\phi_0=\phi(x_0)$, $\phi''_0=\phi''(x_0)$ etc.
Rewrite this as follows
$$
Az^2=Bx^2+Cy^2+z^4h(z)+O((x^2+y^2)^2),
$$
where
\begin{equation}
\begin{split}
A&=\zeta''_0/2=\zeta_0\left(a_3\zeta^2_0+b_3\right),\\
B&=\psi_0\phi''_0/2=\zeta_0\left(c_1\phi_0^2-b_1\right),\\
C&=\phi_0\psi''_0/2=\zeta_0\left(c_2\psi^2_0-b_2\right).
\label{eq45}
\end{split}
\end{equation}
Applying again Theorem \ref{th41} we find
$$
\phi_0^2=(b_1+\delta\sqrt{\Delta(A)})/c_1, \quad
\psi_0^2=(b_2+\delta\sqrt{\Delta(A)})/c_2, \quad
$$
where $\delta^2=1$. By using the generating condition
$
c_1c_2=a_3\beta_3=a_3(b_1+b_2),
$
we find
\begin{equation*}
\begin{split}
\zeta_0^2=\phi^2_0\psi^2_0 &
=\frac{1}{a_3\beta_3}(b_1b_2+\beta_3\delta\sqrt{\Delta(A)}+\Delta(A))
=\frac{b_1b_2+\Delta(A)}{a_3\beta_3}+\frac{\delta}{a_3}\sqrt{\Delta(A)}.
\end{split}
\end{equation*}
By Lemma~\ref{lem32} we have $\Delta(A)=-b_1b_2-b_3\beta_3$, which yields
$$
\zeta_0^2=\frac{-b_3+\delta\sqrt{\Delta(A)}}{a_3}.
$$
Substitution of the found three relations into~(\ref{eq45}) yields
$$
A=B=C=\zeta_0\delta\sqrt{\Delta(A)}\neq 0,
$$
which finally implies
$$
z^2=x^2+y^2+\frac{1}{A}z^4h(z)+O((x^2+y^2)^2),
$$
and the required asymptotic (\ref{assym}) easily follows.
\qed
\end{proof}

\section{Classification of the solutions with isolated singular points}
\label{raz5}

Here we study the local behavior of the gradient of $u(x,y)$
in a small neighborhood of an isolated singularity $m_0$. We show that there can only occur the following three types of singular points:

\begin{itemize}
\item
the 1st type: solution is space-like in a punctured neighborhood of $m_0$;

\item
the 2nd type: any small neighborhood of $m_0$ contains one  space-like and
one  time-like component;

\item
the 3rd type: any small neighborhood of $m_0$ four alternating components: two  space-like and
two  time-like.
\end{itemize}

In order to establish this classification we study the level set $|\nabla u(x,y)|=1$ in a small neighborhood of an isolated nonremovable singular point $(0,0,0)$, where  $z=u(x,y)$ is  the solution~(\ref{eq1}) generated by a matrix $A\in \ZM_3$.
We have
$$1=|\nabla u|^2=\frac{\phi^2\psi'{}^2+\phi'{}^2\psi^2}{\zeta'{}^2},$$
hence
\begin{equation}
\phi'{}^2(x)\psi^2(y)+\phi^2(x)\psi'{}^2(y)-\zeta'{}^2(z)=0.
\label{eq51}
\end{equation}
In order to analyze the latter equation we need the Taylor series for $\phi$, $\psi$, $\zeta$ and its
derivatives. We consider the general function $g(\xi)$ satisfying $g'{}^2(\xi)=P(g^2(\xi))$, where $P(t)=a+2bt+ct^2$ has positive discriminant, and find the Taylor series of $g$ at the point $0$ where $g'(0)=0$. Denote $g_0=g(0)$. Then $g_0^2=(-b+\delta\sqrt{\Delta})/c$,
where $\Delta=b^2-ac$. Then
$$
g_0''=2(b+cg_0^2)g_0=2\delta g_0\sqrt{\Delta}=2\mu g_0,
$$
where $\mu=\delta\sqrt{\Delta(A)}$. We find the higher derivatives $g^{(k)}_0$:
\begin{equation*}
\begin{split}
g_0^{(3)}=2g_0'(b+3cg_0^2)=0,\qquad
g_0^{(4)}=4\mu g_0(b+3cg_0^2).
\end{split}
\end{equation*}
Thus
$$
g(\xi)=g_0+\mu g_0 \xi^2+\frac{g_0\mu(b+3cg_0^2)}{6}\xi^4+o(\xi^4),
$$
hence
\begin{equation}
{g^2(\xi)}=g_0^2\left(1+2\mu \xi^2+\frac{3\mu^2+\mu b+3\mu cg_0^2}{3}\xi^2\right) + o(\xi^4)
\label{eq52}
\end{equation}
By using $g'{}^2(\xi)=P(g^2(\xi))$ we have for the derivative
\begin{equation}
g'{}^2(\xi)=P(g_0^2)+P'(g_0^2)\rho+\frac{1}{2}P''(g_0^2)\rho^2+o(\rho^2),
\label{eq53}
\end{equation}
where $\rho\equiv g^2(\xi)-g_0^2$. Applying the found formulae we obtain
$$
g'{}^2(\xi)=4\mu^2g_0^2\xi^2+
g_0^2\left(2\mu^3+\frac{2}{3}\mu^2(b+9cg_0^2)\right)\xi^4 + o(\xi^4).
$$
Notice that $b+9cg_0^2 = -8b+9\delta\sqrt{\Delta(A)} = 9\mu - 8b$.
Hence,
\begin{equation}
g'{}^2(\xi) = 4\mu^2g_0^2\xi^2\left(1+2(\mu-\frac{2}{3}b)\xi^2\right) + o(\xi^4).
\label{eq54}
\end{equation}
Now  applying (\ref{eq52}) and (\ref{eq54}) to $\phi$, $\psi$ and $\zeta$, and substituting then the resulting formulas into equation (\ref{eq51}), we find after simplification and using $\phi_0^2\psi_0^2=\zeta_0^2$
\begin{equation*}
(1+2\mu x^2+2y^2(\mu+\frac{2}{3}b_2))y^2 +
(1+2\mu y^2+2x^2(\mu+\frac{2}{3}b_1))x^2 -
(1+2z^2(\mu-\frac{2}{3}b_3))z^2 = H,
\end{equation*}
where $H$ contains the terms of order higher than 4.
By Theorem~\ref{th42} we have
$$z^2=x^2+y^2+O((x^2+y^2)^{3/2}),$$
which yields the infinitesimal equation for the gradient level set:
$$
b_3(x^2+y^2)^2+b_2y^4+b_1x^4=0.
$$
We rewrite this in the new notation  $\xi=y^2/x^2$ as
\begin{equation}
\beta_1\xi^2+2b_3\xi+\beta_2=0,
\label{eq56}
\end{equation}
where $\beta_1=b_2+b_3$ and $\beta_2=b_1+b_3$. Notice that the discriminant of the quadratic
equation~(\ref{eq56}) is positive:
$$
4(b_3^2-\beta_1\beta_2)=4\Delta(A)>0.
$$
Hence  the infinitesimal equation of the level lines $|\nabla
z(x,y)|=1$ takes the form
\begin{equation}
(y^2-\xi_1x^2)(y^2-\xi_2x^2)=0,
\label{eq57}
\end{equation}
where $\xi_1$, $\xi_2$ are the roots of the quadratic equation~(\ref{eq56}).

Thus we obtain the following classification.

\begin{table}[h]
\begin{tabular}{|p{110px}|p{110px}|p{110px}|}\hline
1st type&2nd type&3rd type\\\hline
$\beta_1>0,\; \beta_2>0,\; b_3>0$ & $\beta_1<0,\ \beta_2>0$ &
$\beta_1>0,\; \beta_2>0,\; b_3<0$\\\hline
$\beta_1<0,\; \beta_2<0,\; b_3<0$& &$\beta_1<0,\ \beta_2<0,\ b_3>0$
\\\hline
\end{tabular}
\bigskip
\label{tabb}\caption{Three types of singularities}
\end{table}

\section{Examples of the two-periodic solutions}
\label{raz6}

Here we illustrate the above construction by some examples of the surfaces of each type according to the Table~\ref{tabb}. We denote by $\sn(t;k)$ the Jacobi elliptic sinus of the parameter $k$ defined as
$$
\int_{0}^{\sn(t;k)}\frac{du}{\sqrt{(1-u^2)(1-k^2u^2)}}=t,
$$
and by $\cn(t;k)$ Jacobi cosinus satisfying the relation
$\cn(t;k)=\sqrt{1-\sn^2(t;k)}$. We use the standard convention: $k'=\sqrt{1-k^2}$.
The examples below have discriminant of $\Delta(A)= 1/4$.

\noindent
\textbf{A solution of the 1st type}:
\begin{equation}\label{eq:cncn}
\sn\left(\lambda z;km\right)=
\cn\left(x;\frac{k}{\sqrt{1+k^2}}\right)\cn\left(y;
\frac{m}{\sqrt{1+m^2}}\right),
\end{equation}
where $\lambda=1/(km)'$, $k,m>0$ and $0<km<1$. The denerative matrix
has the form
$$
\left(
\begin{array}{ccc}
\frac{1}{1+k^2} & -\frac{(1+k^2)m^2\lambda^2}{(1+m^2)} & -\frac{k^2}{1+k^2}\\
\frac{1}{1+m^2} & -\frac{(1+m^2)k^2\lambda^2}{(1+k^2)} & -\frac{m^2}{1+m^2}\\
k^2m^2\lambda^2 & \frac{1}{\lambda^2(1+k^2)(1+m^2)} & \lambda^2
\end{array}
\right).
$$

\noindent
\textbf{A space-like surface of the 2nd type}:
\begin{equation}\label{eq:sncn}
\cn(z;\mu km) = \sn\left(\frac{x}{k'};k\right)\cn(y;m),
\end{equation}
where $\mu=1/(k'm)'$, $k,m\in(0,1)$.

The generating matrix
$$
\left(
\begin{array}{ccc}
\frac{1}{k'{}^2} & -(k'm'm\mu)^2 & \frac{k^2}{k'{}^2}\\
m'{}^2 & \frac{k^2\mu^2}{k'{}^2} & -m^2\\
-k^2m^2\mu^2 & \frac{1}{k'{}^2\mu^2} & m'{}^2\mu^2
\end{array}
\right).
$$

\noindent
\textbf{A surface of the 3rd type}:
\begin{equation}\label{eq:snsn}
\sn\left(\lambda z;km\right)=\sn\left(\frac{x}{k'};k\right)
\sn\left(\frac{y}{m'};m\right),
\end{equation}
where $\lambda=1/(km)'$, $k,m\in(0,1)$. The generating matrix
$$
\left(
\begin{array}{ccc}
\frac{1}{k'{}^2} & \frac{\lambda^2k'{}^2m^2}{m'{}^2} & \frac{k^2}{k'{}^2}\\
\frac{1}{m'{}^2} & \frac{\lambda^2k^2m'{}^2}{k'{}^2} & \frac{m^2}{m'{}^2}\\
\lambda^2k^2m^2 & \frac{1}{\lambda^2k'{}^2m'{}^2} & \lambda^2
\end{array}
\right).
$$

\begin{figure}[h]
\includegraphics[height=0.37\textheight,angle=270,keepaspectratio=true]{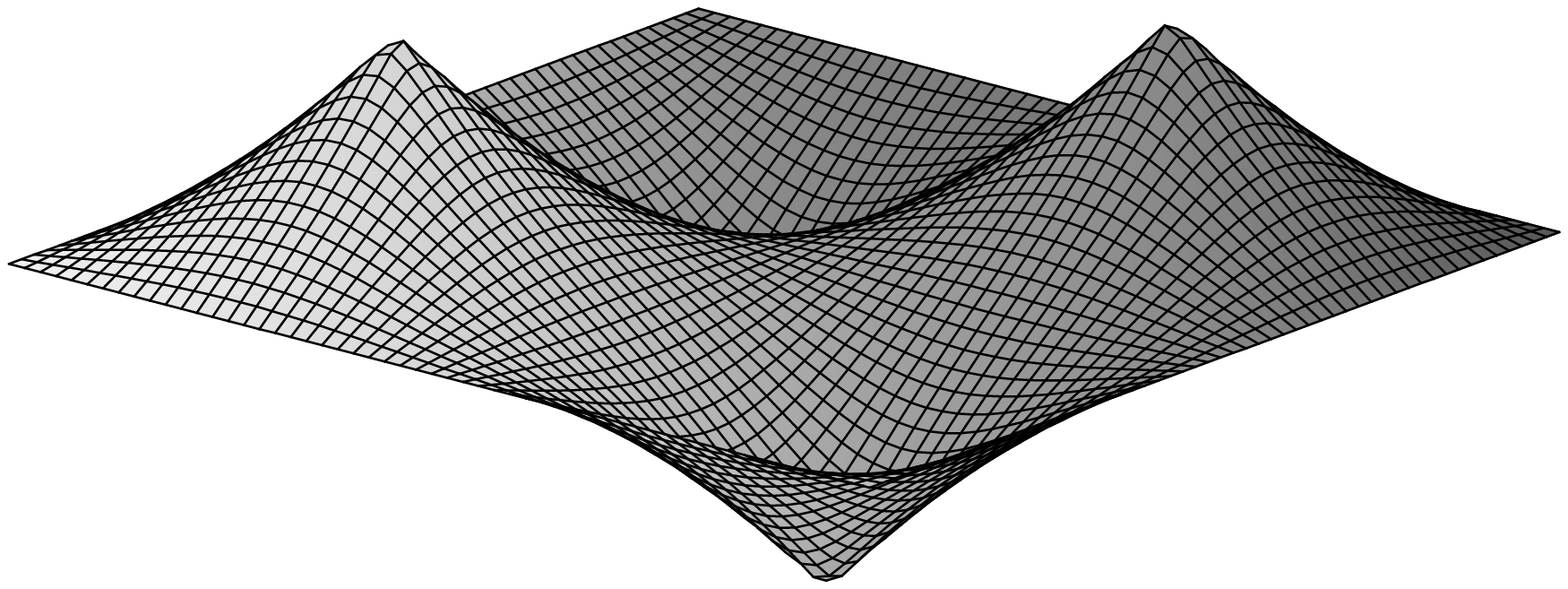}
\includegraphics[height=0.29\textheight,angle=270,keepaspectratio=true]{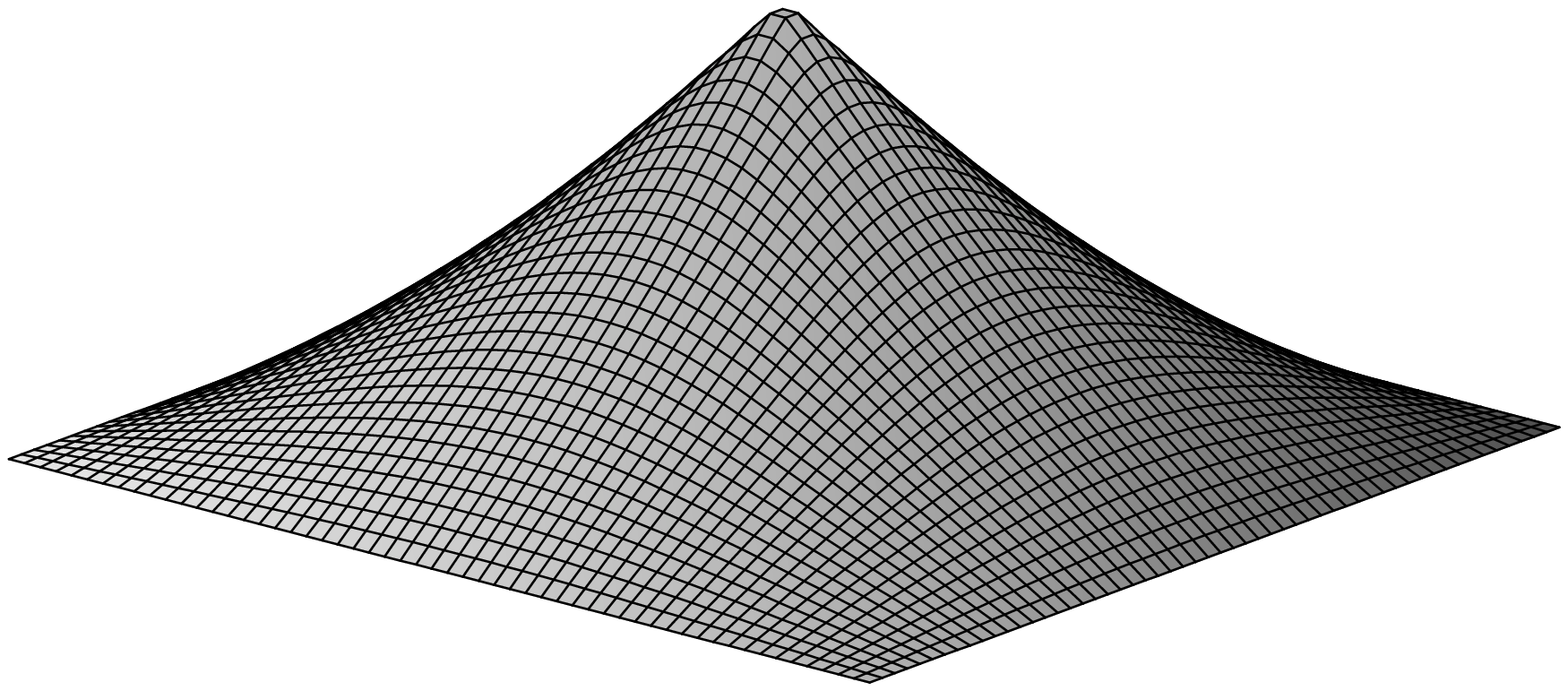}
\vspace*{-0.9cm}
\caption{A surface (\ref{eq:cncn}) for $k=4/5$ and its fundamental domain} \label{fig:snsn}
\end{figure}

\begin{figure}[h]
\includegraphics[height=0.37\textheight,angle=270,keepaspectratio=true]{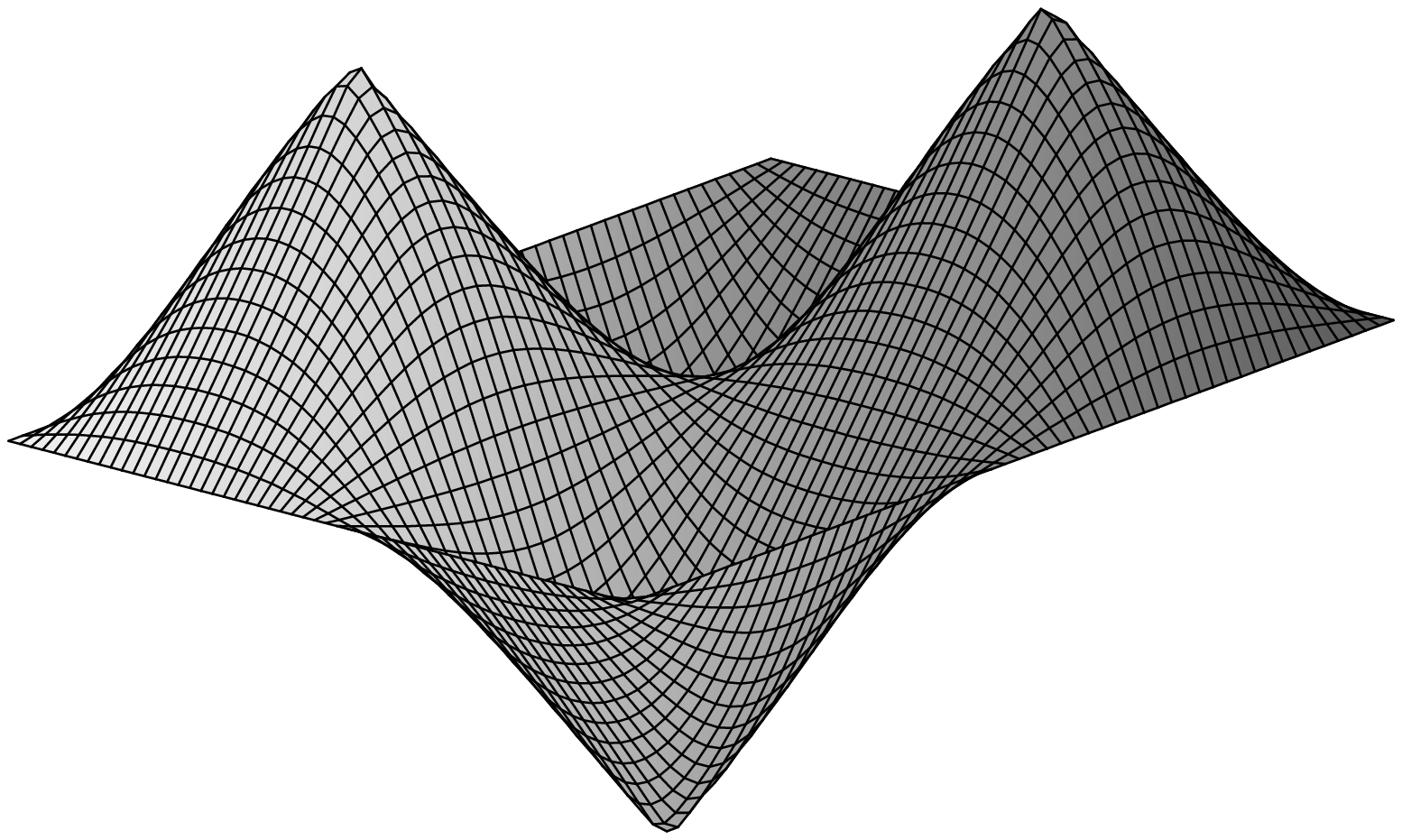}
\includegraphics[height=0.29\textheight,angle=270,keepaspectratio=true]{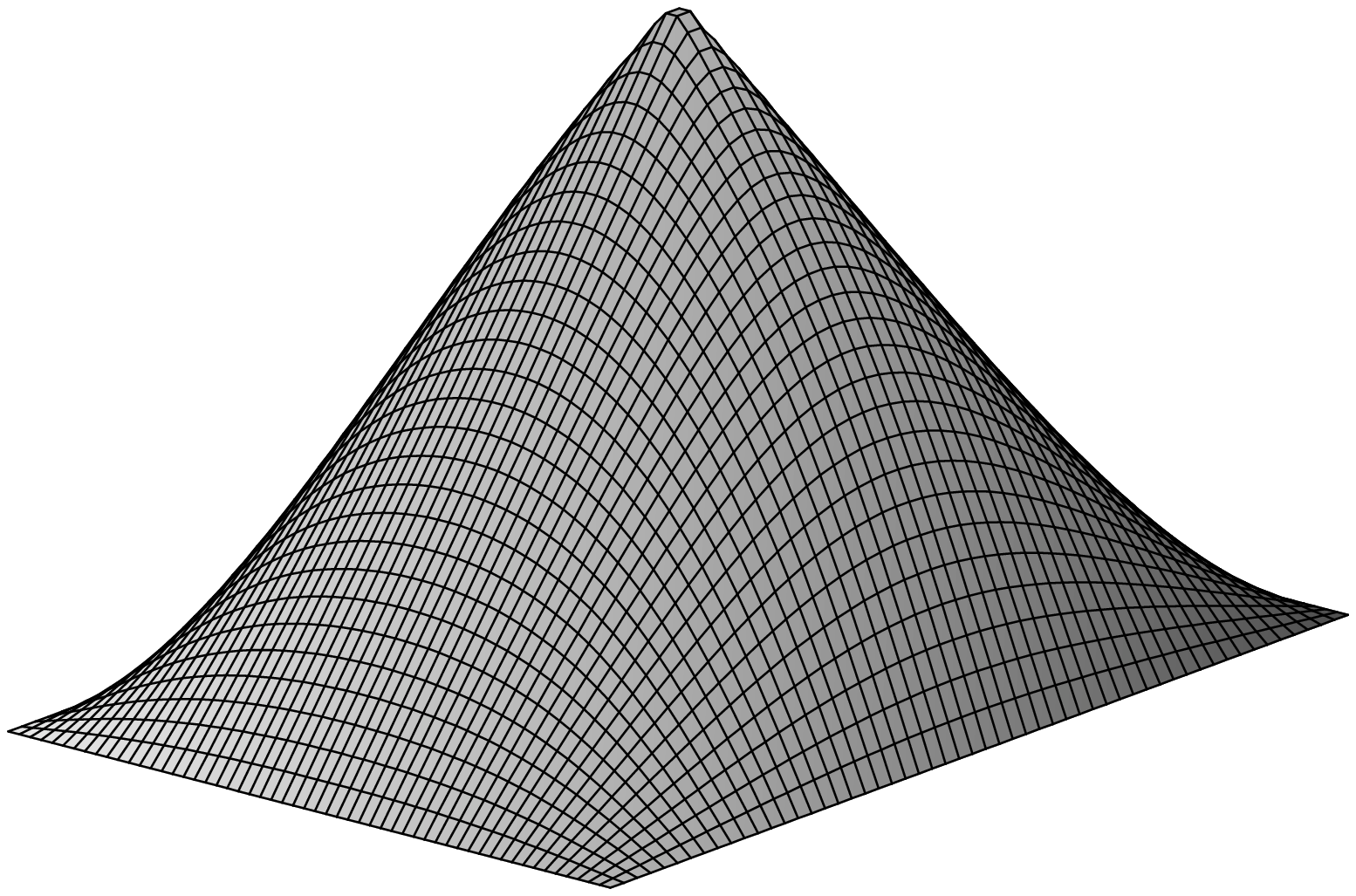}
\vspace*{-0.9cm}
\caption{A surface (\ref{eq:sncn}) for $k=4/5$ and its fundamental domain} \label{fig:snsn}
\end{figure}

\begin{figure}[h]
\includegraphics[height=0.37\textheight,angle=270,keepaspectratio=true]{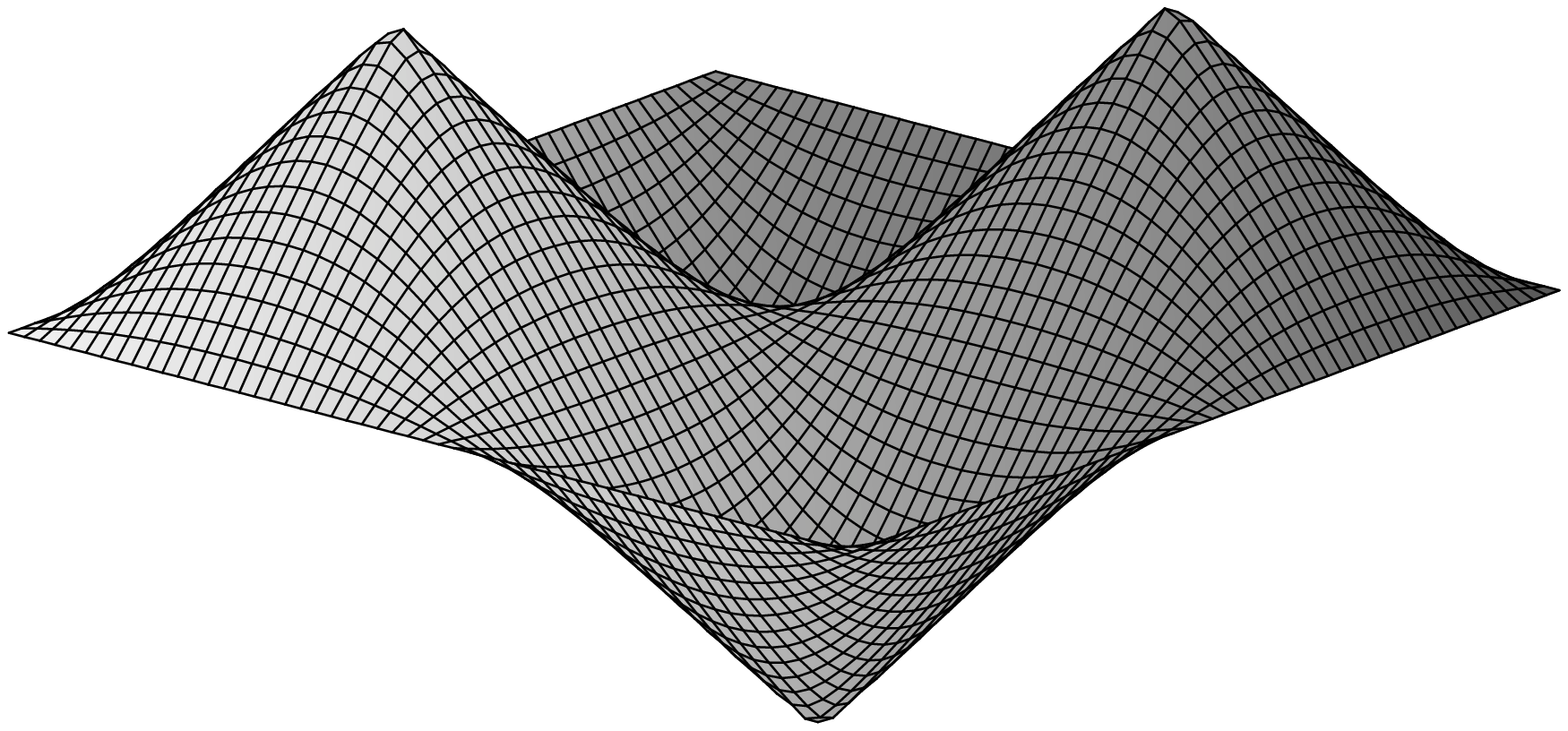}
\includegraphics[height=0.29\textheight,angle=270,keepaspectratio=true]{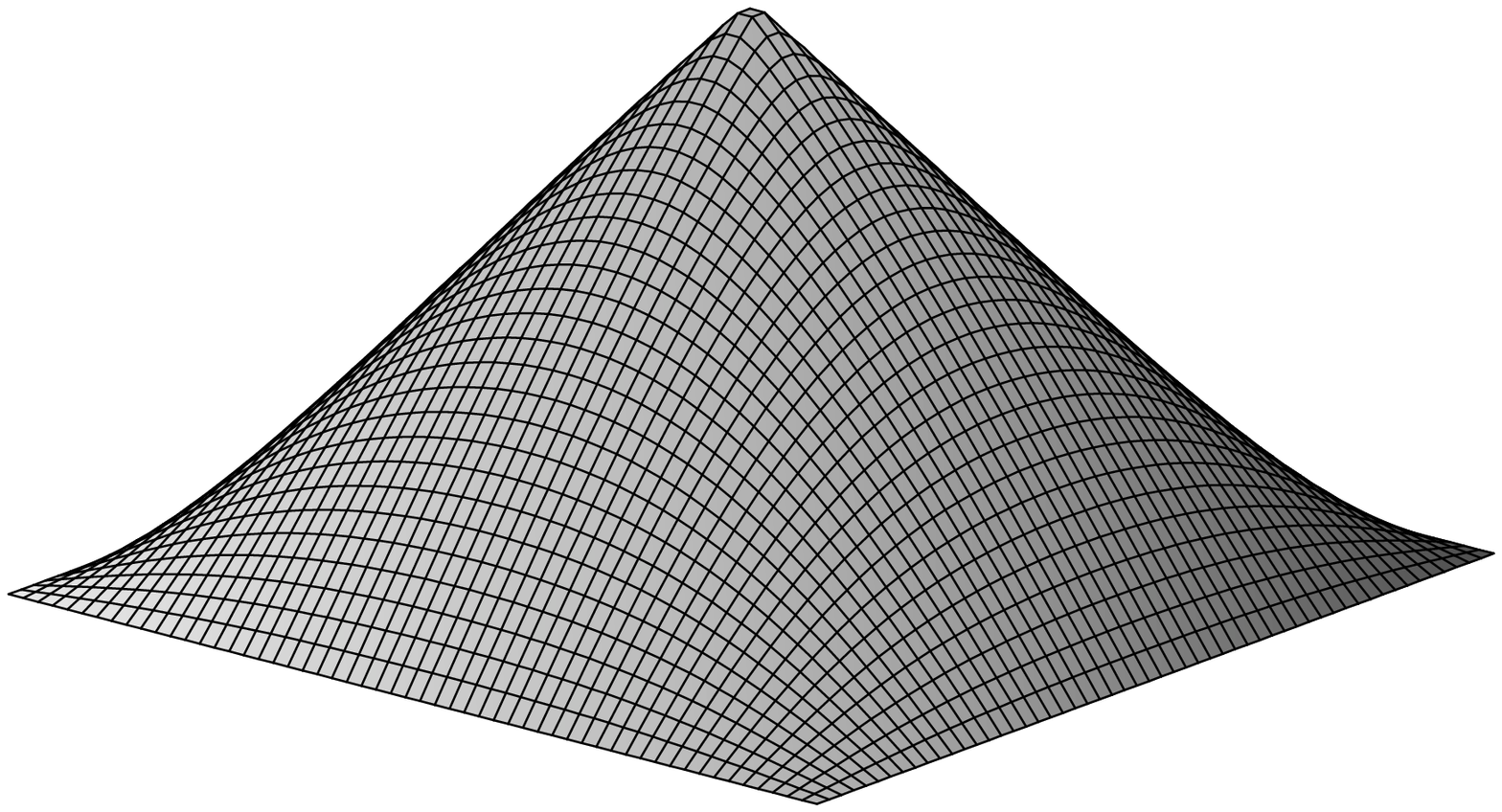}
\vspace*{-0.9cm}
\caption{A surface (\ref{eq:snsn}) for $k=4/5$ and its fundamental domain} \label{fig:snsn}
\end{figure}

\bibliographystyle{amsunsrt}

\end{document}